\documentclass[a4paper,12pt,reqno]{amsart}

\usepackage{amsmath}
\usepackage{amssymb}
\usepackage{amsfonts}
\usepackage{graphicx}
\usepackage[colorlinks]{hyperref}
\renewcommand\eqref[1]{(\ref{#1})} 

\graphicspath{ {images/} }
\title[Geometric Hardy inequalities]{Geometric Hardy inequalities on  starshaped sets}
\author[Michael Ruzhansky]{Michael Ruzhansky}
\address{\href{www.ruzhansky.org}{Michael Ruzhansky:}
	\endgraf
	Department of Mathematics: Analysis, Logic and Discrete Mathematics
	\endgraf
	Ghent University, Belgium
	\endgraf
	and
	\endgraf
	School of Mathematical Sciences
		\endgraf Queen Mary University of London 
			\endgraf
		United Kingdom
			\endgraf
	{\it E-mail address} {\rm Michael.Ruzhansky@ugent.be}
}

\author[Bolys Sabitbek]{Bolys Sabitbek}
\address{ \href{https://www.researchgate.net/profile/Bolys_Sabitbek3}{Bolys Sabitbek:}
	\endgraf
	Institute of Mathematics and Mathematical Modeling 
	\endgraf
	125 Pushkin Street., Almaty, 050010
	\endgraf
	Kazakhstan
	\endgraf
	and
	\endgraf 
	Department of Mechanics and Mathematics
	\endgraf 
	Al-Farabi Kazakh National University 
	\endgraf
	71 al-Farabi Ave., Almaty, 050040 
	\endgraf Kazakhstan
	\endgraf
	{\it E-mail address} {\rm b.sabitbek@math.kz}
}

\author[Durvudkhan Suragan]{Durvudkhan Suragan}
\address{\href{https://sst.nu.edu.kz/en/durvudkhan-suragan-phd/}{Durvudkhan Suragan:}
	\endgraf
	Department of Mathematics
	\endgraf
	Nazarbayev University
	\endgraf
	53 Kabanbay batyr Ave., Astana, 010000
	\endgraf
	Kazakhstan
	\endgraf
	{\it E-mail address} {\rm durvudkhan.suragan@nu.edu.kz}
}


\thanks{The first
	author was supported by the EPSRC Grant 
	EP/R003025/1, by the Leverhulme Research Grant RPG-2017-151, and by the FWO Odysseus grant. The second author was supported by Nazarbayev University Faculty Development Competitive Research Grants N090118FD5342. The third author was supported in parts by the MESRK grant AP05130981. No new data was collected or generated during the course of this research.}

\newtheoremstyle{theorem}
{10pt}          
{10pt}  
{\sl}  
{\parindent}     
{\bf}  
{. }    
{ }    
{}     
\theoremstyle{theorem}
\newtheorem{theorem}{Theorem}
\newtheorem{corollary}[theorem]{Corollary}
\numberwithin{equation}{section}
\theoremstyle{plain}
\newtheorem{thm}{Theorem}[section]

\theoremstyle{definition}
\newtheorem{defn}[thm]{Definition}
\newtheorem{rem}[thm]{Remark}

\newtheoremstyle{defi}
{10pt}          
{10pt}  
{\rm}  
{\parindent}     
{\bf}  
{. }    
{ }    
{}     
\theoremstyle{defi}

\newtheorem{remark}[theorem]{Remark}

\begin{document}
		\begin{abstract}
		In this paper, we present the geometric Hardy inequalities on the starshaped sets in the Carnot groups. Also, we obtain the geometric Hardy inequalities on half-spaces for general vector fields.  
	\end{abstract}
	\maketitle
\section{Introduction}\label{sec1}
In 1998, Danielli and Garofalo \cite{DG98} firstly introduced the concept of starshapedness on the Carnot groups (see also \cite{DG}). Their paper provides the geometrical properties of starshaped and convex sets. The convexity in the Heisenberg groups was studied by many authors such as Monti and Rickly \cite{MR} who proved the geodesic convexity, or by Danielli, Garofalo, and Nhieu \cite{DGN} (see also \cite{Garofalo}) who introduced the concept of horizontal convexity ($H$-convexity). Bardi and Dragoni \cite{BD3}, \cite{BD4} generalised the concept of convexity to general vector fields and introduced the notion of $\mathcal{X}$-convexity which is a generalisation of $H$-convexity. This analysis allows introducing the distance to the boundary notation for starshaped sets, so by using the distance formula one can obtain geometric Hardy type inequalities.

{\bf Acknowledgment:} We thank Nicola Garofalo for bringing to our attention the paper \cite{DGS} and also for kindly sharing with us the proper definition of starshapedness. 

\smallskip
The main aim of this paper is to obtain the geometric Hardy inequalities on starshaped sets in the Carnot groups. Moreover, we present the geometric Hardy inequalities on the half-spaces for general vector fields.

We organise the paper in the following way: 
\begin{itemize}
	\item[Sec. \ref{sec1}:] We give a brief overview of the sub-Riemannian manifolds, Grushin plane, Carnot groups, Heisenberg groups, and Engel groups.
		\item[Sec. \ref{sec4}:]  We obtain the geometric Hardy inequalities on the starshaped sets in the Carnot groups and provide some examples.
	\item[Sec. \ref{sec3}:] We obtain the geometric Hardy inequalities on the half-spaces for general vector fields and provide some examples.

	\item[Sec. \ref{sec5}:] We give the proofs of main results.
\end{itemize}
  
\subsection{Sub-Riemannian manifolds}
	Let $M$ be a smooth manifold of dimension $n$ with a family of vector fields $\{X_k\}_{k=1}^N$, $n\geq N$, defined on $M$ satisfying the H\"ormander rank condition. Then they induce a sub-Riemannian metric $\langle \cdot, \cdot\rangle_{\mathcal{H}}$ on the associated space $\mathcal{H}_x= {\rm span}(X_1(x),\ldots,X_N(x))$.  The triple $(M,\mathcal{H},\langle \cdot, \cdot\rangle_{\mathcal{H}})$ is a so-called sub-Riemannian manifold (with sub-Riemannian geometry). Note that, unlike for Carnot groups, in general, it is not possible to define dilations, translations, the homogeneous norm and the distance on sub-Riemannian manifolds. 
	
	Let us denote the operator of the sum of squares of vector fields by 
	\begin{equation}
		\mathcal{L}:= \sum_{k=1}^{N} X_k^2.
	\end{equation}
These operators have been studied by many authors, for instance, it is well-known since H\"ormander's pioneering work \cite{Hormander67} that if the commutators of the vector fields $\{X_k\}_{k=1}^N$ generate the Lie algebra, the operator $\mathcal{L}$ is locally hypoelliptic. The
$p$-version of the sum of squares of vector fields can be given by the formula
\begin{equation}
	\mathcal{L}_p f:= \nabla_{X} \cdot (|\nabla_{X} f|^{p-2}\nabla_{X} f),
\end{equation}
where 
\begin{equation*}
	\nabla_{X}:= (X_1,\ldots,X_N).
\end{equation*}
\subsection{Grushin plane}
One of the important examples of a sub-Riemannian manifold is the Grushin plane. The Grushin plane is the space $\mathbb{R}^2$ with vector fields 
\begin{equation*}
	X_1 = \frac{\partial}{\partial x_1}, \,\, \text{and} \,\, X_2 = x_1\frac{\partial}{\partial x_2},
\end{equation*}
for $x:=(x_1,x_2)\in \mathbb{R}^2$.
\subsection{Carnot groups}
Let $\mathbb{G}=(\mathbb{R}^n,\circ,\delta_{\lambda})$ be a stratified Lie group (or a homogeneous Carnot group or just a Carnot group), with the dilation structure $\delta_{\lambda}$ and Jacobian generators $X_{1},\ldots,X_{N}$, so that $N$ is the dimension of the first stratum of $\mathbb{G}$. Let us denote by $Q$ the homogeneous dimension of $\mathbb{G}$.  We refer to the recent books \cite{FR} and \cite{RS_book} for extensive discussions of stratified Lie groups and their properties.

The sub-Laplacian on $\mathbb{G}$ is given by
\begin{equation}\label{sublap}
\mathcal{L}=\sum_{k=1}^{N}X_{k}^{2}.
\end{equation}
We also recall that the standard Lebesgue measure $dx$ on $\mathbb R^{n}$ is the Haar measure for $\mathbb{G}$ (see, e.g. \cite[Proposition 1.6.6]{FR}).
Each left invariant vector field $X_{k}$ has an explicit form and satisfies the divergence theorem,
see e.g. \cite{FR} for the derivation of exact formula: more precisely, we can express
\begin{equation}\label{Xk0}
X_{k}=\frac{\partial}{\partial x'_{k}}+
\sum_{l=2}^{r}\sum_{m=1}^{N_{l}}a_{k,m}^{(l)}(x',...,x^{(l-1)})
\frac{\partial}{\partial x_{m}^{(l)}},
\end{equation}
with $x=(x',x^{(2)},\ldots,x^{(r)})$, where $r$ is the step of $\mathbb{G}$ and
$x^{(l)}=(x^{(l)}_1,\ldots,x^{(l)}_{N_l})$ are the variables in the $l^{th}$ stratum,
see also \cite[Section 3.1.5]{FR} for a general presentation.
The horizontal divergence is defined by
$${\rm div}_{H} f:=\nabla_{H}\cdot f,$$
where 
$$\nabla_{H}:=(X_{1},\ldots, X_{N})$$
is the horizontal gradient.
The $p$-sub-Laplacian has the form
\begin{equation}\label{Lp}
\mathcal{L}_p f = \nabla_{H} \cdot (|\nabla_{H} f|^{p-2}\nabla_{H} f).
\end{equation}

\subsection{Heisenberg groups}

Let $\mathbb{H}_1$ be the Heisenberg group, that is, the set $\mathbb{R}^{3}$ equipped with the group law 
\begin{equation*}
x \circ x' := (x_1 + x'_1, x_2 + x'_2, x_3 + x_3'+2(x_1' x_2 - x_1 x_2')),
\end{equation*}
where $x:= (x_1,x_2,x_3) \in \mathbb{R}^3$, and $x^{-1}=-x$ is the inverse element of $x$ with respect to the group law. The dilation operation on the Heisenberg group with respect to the group law has the form
\begin{equation*}
\delta_{\lambda}(x) := (\lambda x_1, \lambda x_2, \lambda^2 x_3) \,\, \text{for}\,\, \lambda>0.
\end{equation*} 
The Lie algebra $\mathfrak{h}$ of the left-invariant vector fields on the Heisenberg group $\mathbb{H}_1$ is spanned by 
\begin{equation*}
X_1:= \frac{\partial }{\partial x_1} + 2x_2\frac{\partial }{\partial x_3},
\end{equation*}
\begin{equation*}
X_2:= \frac{\partial }{\partial x_2} - 2x_1\frac{\partial }{\partial x_3},
\end{equation*}
with their (non-zero) commutator
\begin{equation*}
[X_1,X_2]= - 4 \frac{\partial}{\partial x_3}.
\end{equation*} 
The horizontal gradient on $\mathbb{H}_1$ is given by 
\begin{equation*}
\nabla_{H}:= (X_1,X_2),
\end{equation*} 
so the sub-Laplacian on $\mathbb{H}_1$ is given by
\begin{equation*}
\mathcal{L}:=X_1^2 + X_2^2. 
\end{equation*} 
The Heisenberg group is the most common example of a step 2 stratified group (Carnot group).
\subsection{Engel groups}
Let $\mathbb{E}$ be the Engel group, that is, the set $\mathbb{R}^{4}$ equipped with the group law 
\begin{equation*}
x \circ x' := (x_1 + x'_1, x_2 + x'_2, x_3+x_3' + P_3, x_4+x'_4 +P_4),
\end{equation*}
where 
\begin{align*}
P_3 &= \frac{1}{2}(x_1x_2'-x_2x_1'),\\
P_4&= \frac{1}{2}(x_1x_3'-x_3x_1')+\frac{1}{12}(x_1^2x_2'-x_1x_1'(x_2+x_2')+x_2x_1'^2).
\end{align*}
Here $x:=(x_1,x_2,x_3,x_4)\in \mathbb{R}^4$. The vector fields have the following form
\begin{align*}
X_1:&= \frac{\partial}{\partial x_1} - \frac{x_2}{2}\frac{\partial}{\partial x_3}- \left(\frac{x_3}{2}+ \frac{x_1x_2}{12}\right)\frac{\partial}{\partial x_4},\\
X_2:&= \frac{\partial}{\partial x_2} + \frac{x_1}{2}\frac{\partial}{\partial x_3}+\frac{x_1^2}{12}\frac{\partial}{\partial x_4},\\
X_3:&=[X_1,X_2]= \frac{\partial}{\partial x_3} + \frac{x_1}{2} \frac{\partial}{\partial x_4}, \\
X_4:&=[X_1,X_3]=\frac{\partial }{\partial x_4}.	
\end{align*}
The Engel group is a well-known example of a step 3 stratified group (Carnot group).

\section{Hardy inequalities on starshaped sets}\label{sec4}

In order to present the results on the starshaped domains, let us recall the definition of starshaped sets in a Carnot group $\mathbb{G}=(\mathbb{R}^n,\circ,\delta_{t})$ and related arguments.
\begin{defn}[Starshapedness \cite{DG98}]
	Let $\Omega \subset \mathbb{G}$ be a $C^1$ domain containing the identity $e$. Then $\Omega$ is starshaped with respect to $e$ if for every $x \in \partial \Omega$ one has 
	\begin{equation}
	\langle Z(x), n(x) \rangle \geq 0,
	\end{equation}
	where $n$ is the Riemannian outer normal to $\partial \Omega$. 
	
	When the strict inequality holds, then $\Omega$ is said to be strictly starshaped with respect to $e$.
\end{defn}
Here the vector fields $Z$ are the infinitesimal generator of this group automorphism. This vector fields $Z$ takes the form
\begin{equation}\label{Z_vector}
Z = \sum_{i=1}^{N} x'_i \frac{\partial}{\partial x'_i} + 2 \sum_{l=1}^{N_2} x_{2,l} \frac{\partial}{\partial x_{2,l}} + \cdots + r \sum_{l=1}^{N_r} x_{r,l}  \frac{\partial}{\partial x_{r,l}}.  
\end{equation}   
Then for $x' \in \mathbb{R}^N$ and $x^{(i)}\in \mathbb{R}^{N_i}$ with $i=2,\ldots,r$ we have 
\begin{equation}
Z(x) = (x', 2x^{(2)}, \cdots, rx^{(r)}),
\end{equation}
and 
\begin{align*}
\langle Z(x), n(x)\rangle =& x'n' + 2 x^{(2)} n^{(2)} + \ldots + r x^{(r)}n^{(r)}\\
=& x'_1n'_1 + \cdots+x'_Nn'_N+ 2(x_{2,1}n_{2,1}+\cdots+x_{2,N_2}n_{2,N_2})\\
&+ \cdots + r(x_{r,1}n_{r,1}+\cdots+x_{r,N_r}n_{r,N_r}),
\end{align*}
since $n(x):=(n',n^{(2)},\ldots,n^{(r)})$ with $n'\in \mathbb{R}^N$ and $n^{(i)} \in \mathbb{R}^{N_i}, \,\, i=2,\ldots,r$.

Based on the above arguments now we present the geometric Hardy inequalities on the starshaped sets for the sub-Laplacians.
\begin{theorem}\label{G_starshaped}
	Let $\Omega$ be a starshaped set on a Carnot group. Then for every $\gamma \in \mathbb{R}$ and $p>1$ we have the following Hardy inequality
	\begin{align}\label{Hardy_stratified_starshaped}
	\int_{\Omega} |\nabla_H f(x)|^p dx \geq& -(p-1)(|\gamma|^{\frac{p}{p-1}}+\gamma) \int_{\Omega} \frac{|\nabla_H \langle Z(x),n(x)\rangle|^p}{|\langle Z(x),n(x)\rangle|^p}|f(x)|^p dx\\
	& + \gamma \int_{\Omega} \frac{\mathcal{L}_p(\langle Z(x),n(x)\rangle)}{|\langle Z(x), n(x)\rangle|^{p-1}} |f(x)|^p dx, \nonumber
	\end{align}
	for every function $f\in C_0^{\infty}(\Omega)$. 
\end{theorem}
\begin{corollary}\label{cor234}
	Let $\mathbb{H}^*$ be a starshaped set on the Heisenberg group $\mathbb{H}_1$. Then for $p>1$, we have the following Hardy inequality
	\begin{equation}\label{eqd}
	\int_{\mathbb{H}^*}  |\nabla_H f(x)|^p dx \geq \left( \frac{p-1}{p}\right)^p \int_{\mathbb{H}^*} \frac{ |(n_1+4x_2n_3,n_2-4x_1n_3)|^p}{|x_1n_1+x_2n_2+2x_3n_3|^p}|f(x)|^pdx,
	\end{equation}
	for every function $f\in C_0^{\infty}(\mathbb{H}^*)$.
\end{corollary}
\begin{remark}
	Note that in the case $$\mathbb{H}^*:=\{ \langle Z(x), n(x) \rangle>0, \,\, \forall x\in \partial \mathbb{H}^*, \, Z(x):=(x_1,x_2,2x_3) \}=\{x \in \mathbb{H}_1\cong \mathbb{R}^3: x_3>0 \}$$ with $n(x):=(0,0,1)$, and $p=2$, we have the inequality
	\begin{equation*}
	\int_{\mathbb{H}^*}  |\nabla_H f(x)|^2 dx \geq  \int_{\mathbb{H}^*} \frac{|x_1|^2+|x_2|^2}{|x_3|^2}|f(x)|^2 dx.
	\end{equation*}
\end{remark}
\begin{proof}[Proof of Corollary \ref{cor234}]
	We begin the proof of Corollary \ref{cor234} by a simple computation such as
	\begin{align*}
	\langle Z(x),n(x)\rangle &= x_1n_1 + x_2n_2 + 2x_3n_3,\\
	\nabla_{H}	\langle Z(x),n(x)\rangle &= (n_1+4x_2n_3,n_2-4x_1n_3), \\
	|\nabla_{H}	\langle Z(x),n(x)\rangle|^p &= \left( (n_1+4x_2n_3)^2+ (n_2-4x_1n_3)^2 \right)^{p/2},
	\end{align*}
	and 
	\begin{align*}
	\mathcal{L}_p 	\langle Z(x),n(x)\rangle =& \nabla_{H} \cdot (|\nabla_{H}	\langle Z(x),n(x)\rangle|^{p-2}\nabla_{H}	\langle Z(x),n(x)\rangle)  \\
	=& X_1 	(|\nabla_{H}	\langle Z(x),n(x)\rangle|^{p-2}(n_1+4x_2n_3) )\\
	+& X_2	(|\nabla_{H}	\langle Z(x),n(x)\rangle|^{p-2}(n_2-4x_1n_3))	\\
	=& -4(p-2)|\nabla_{H}\langle Z(x),n(x)\rangle|^{p-4}(n_1+4x_2n_3)(n_2-4x_1n_3)n_3 \\
	& + 4(p-2)|\nabla_{H}\langle Z(x),n(x)\rangle|^{p-4} (n_2-4x_1n_3)(n_1+2x_4n_3)n_3 \\
	=& 0.	
	\end{align*}
	Plugging the above expressions into inequality \eqref{Hardy_stratified_starshaped} and maximising with respect to $\gamma$, we arrive at inequality \eqref{eqd} which proves Corollary \ref{cor234}.
\end{proof}
\begin{corollary}\label{cor_E1_starshaped}
	Let $\mathbb{E}^*$ be a starshaped set on the Engel group $\mathbb{E}$. Then for every function $f\in C_0^{\infty}(\mathbb{E}^*)$, $\gamma \in \mathbb{R}$ and $p=2$, we have 
	\begin{align}\label{E_starshaped}
	\int_{\mathbb{E}^*}  |\nabla_H f(x)|^2 dx \geq& - (|\gamma|^2+\gamma) \int_{\mathbb{E}^*} \frac{|\nabla_{H} \langle Z(x),n(x)\rangle|^2}{\langle Z(x),n(x)\rangle^2} |f(x)|^2 dx \\
	& + \frac{\gamma}{2} \int_{\mathbb{E}^*} \frac{x_2n_4}{\langle Z(x),n(x)\rangle}|f(x)|^2 dx. \nonumber
	\end{align}
\end{corollary}
\begin{proof}[Proof of Corollary \ref{cor_E1_starshaped}] 
	We begin the proof of Corollary \ref{cor_E1_starshaped} by a simple computation such as
	\begin{align*}
	\langle Z(x),n(x)\rangle &= x_1n_1+x_2n_2+ 2x_3n_3 +3x_4n_4, \\
	\nabla_{H}	\langle Z(x),n(x)\rangle &= \left( n_1-x_2n_3 - \frac{3x_3n_4}{2}- \frac{x_1x_2n_4}{4}, n_2+x_1n_3+\frac{x_1^2n_4}{4}\right),\\
	|\nabla_{H}\langle Z(x),n(x)\rangle|^2 &= \left(n_1-x_2n_3 - \frac{3x_3n_4}{2}- \frac{x_1x_2n_4}{4}\right)^2+ \left(n_2+x_1n_3+\frac{x_1^2n_4}{4}\right)^2,
	\end{align*}
	and
	\begin{align*}
	\mathcal{L}	\langle Z(x),n(x)\rangle=& \nabla_{H} \cdot \nabla_H\langle Z(x),n(x)\rangle\\
	=& X_1\left( n_1-x_2n_3 - \frac{3x_3n_4}{2}- \frac{x_1x_2n_4}{4}\right)  + X_2\left( n_2+x_1n_3+\frac{x_1^2n_4}{4} \right) \\
	=&  \frac{x_2n_4}{2}.
	\end{align*}
	Plugging the above expressions into inequality \eqref{Hardy_stratified_starshaped} 
	\begin{align*}
	\int_{\mathbb{E}^*}  |\nabla_H f(x)|^2 dx \geq& - (|\gamma|^2+\gamma) \int_{\mathbb{E}^*} \frac{|\nabla_{H} \langle Z(x),n(x)\rangle|^2}{\langle Z(x),n(x)\rangle^2} |f(x)|^2 dx \\
	& + \frac{\gamma}{2} \int_{\mathbb{E}^*} \frac{x_2n_4}{\langle Z(x),n(x)\rangle}|f(x)|^2 dx,
	\end{align*}
	which proves Corollary \ref{cor_E1_starshaped}.
\end{proof}

\section{Hardy inequalities on half-spaces for general vector fields}\label{sec3}
Let us define the half-space of a sub-Riemannian manifold by
\begin{equation*}
\Omega^+:= \{x \in \mathbb{R}^n : \langle x , n(x)\rangle >d \},
\end{equation*}
where $n(x) \in \mathbb{R}^n$ is the Riemannian outer unit normal to $\partial \Omega^+$ and $d \in \mathbb{R}$. 
 The Euclidean distance to the boundary $\partial \Omega^+$ is denoted by $dist(x,\partial \Omega^+)$ and defined by
\begin{equation*}
dist(x,\partial \Omega^+) := \langle x, n(x) \rangle - d.
\end{equation*}
Then we have:
\begin{theorem}\label{M_half-space}
	Let $M$ be a sub-Riemannian manifold, let $\Omega^+ \subset M$ be a half-space and let $X_1,\ldots,X_N$ be the general vector fields. Then for every $\gamma \in \mathbb{R}$ and $p>1$ we have the following Hardy inequality
	\begin{align}\label{Hardy_half-space}
	\int_{\Omega^+} |\nabla_X f(x)|^p dx \geq& -(p-1)(|\gamma|^{\frac{p}{p-1}}+\gamma) \int_{\Omega^+} \frac{|\nabla_X dist(x,\partial \Omega^+)|^p}{dist(x,\partial \Omega^+)^p}|f(x)|^p dx\\
	& + \gamma \int_{\Omega^+} \frac{\mathcal{L}_p(dist(x,\partial \Omega^+)}{dist(x,\partial \Omega^+)^{p-1}} |f(x)|^p dx, \nonumber
	\end{align}
	for every function $f\in C_0^{\infty}(\Omega^+)$. 
\end{theorem}
Note that inequality \eqref{Hardy_half-space} was obtained in the Carnot groups by the authors in \cite{RSS_geo_H-S}, but here we extend it to general sub-Riemannian manifolds.

Let us give examples for the Heisenberg group (step 2), the Engel group (step 3), and the Grushin plane which does not have a group structure, but serves as an important example of the sub-Riemannian geometry. 

\begin{corollary}\label{G_half-space}
	Let $\Omega^+$ be a half-space in the Grushin plane $G$. Then for every function $f\in C_0^{\infty}(\Omega^+)$ and $p>1$, we have the following Hardy inequality

	\begin{align}\label{Hardy_G_half-space}
\int_{\Omega^+} |\nabla_{X} f(x)|^p dx \geq& -(p-1)(|\gamma|^{\frac{p}{p-1}}+\gamma) \int_{\Omega^+} \frac{(n_1^2+x_1^2n_2^2)^{p/2}}{(x_1n_1+x_2n_2-d)^p}|f(x)|^p dx \\
& +(p-2)\gamma\int_{\Omega^+} \frac{|\nabla_{X}	dist(x,\partial \Omega^+)|^{p-4}n_1n_2^2 x_1}{(x_1n_1+x_2n_2-d)^{p-1}} |f(x)|^p dx.  \nonumber
\end{align}
If one of the cases $n(x)=(1,0)$ or $n(x)=(0,1)$ holds, then we have
\begin{equation}\label{Hardy_with_best_const}
	\int_{\Omega^+} |\nabla_{X} f(x)|^p dx \geq \left(\frac{p-1}{p}\right)^p \int_{\Omega^+} \frac{(n_1^2+x_1^2n_2^2)^{p/2}}{(x_1n_1+x_2n_2-d)^p}|f(x)|^p dx,
\end{equation}
where $dist(x,\partial \Omega^+)=\langle x , n(x)\rangle-d$ and $d\in \mathbb{R}$.

\end{corollary}
\begin{remark}
	Note that, with $\nabla_{X}$ the Grushin gradient,
	\begin{itemize}
		\item If $\Omega^+= \{x \in \mathbb{R}^2 : x_1 >d \}$ with $n(x)=(1,0)$, then we have
		\begin{equation*}
		\int_{\Omega^+} |\nabla_{X} f(x)|^p dx \geq \left(\frac{p-1}{p}\right)^p \int_{\Omega^+} \frac{|f(x)|^p}{|x_1-d|^p} dx.
		\end{equation*} 
		\item If $\Omega^+:= \{x \in \mathbb{R}^2 : x_2 >d \}$ with $n(x)=(0,1)$, then we have
		 	\begin{equation*}
		 \int_{\Omega^+} |\nabla_{X} f(x)|^p dx \geq \left(\frac{p-1}{p}\right)^p \int_{\Omega^+} \frac{|x_1|^p}{|x_2-d|^p}|f(x)|^p dx.
		 \end{equation*}
	\end{itemize}
\end{remark}
\begin{proof}[Proof of Corollary \ref{G_half-space}]
	We begin the proof of Corollary \ref{G_half-space} by a simple computation such as
	\begin{align*}
		dist(x,\partial \Omega^+) &=x_1n_1 + x_2n_2-d,\\
		\nabla_{X} 	dist(x,\partial \Omega^+)& = (n_1,x_1n_2),\\
			|\nabla_{X} 	dist(x,\partial \Omega^+)|^p &= (n_1^2 + x_1^2n_2^2)^{p/2},
	\end{align*} 
and
	\begin{align*}
		\mathcal{L}_p	dist(x,\partial \Omega^+)&= \nabla_{X} \cdot (|\nabla_{X} 	dist(x,\partial \Omega^+)|^{p-2}\nabla_{X} 	dist(x,\partial \Omega^+)) \\
		&=\frac{\partial}{\partial x_1}((n_1^2 + x_1^2n_2^2)^{\frac{p-2}{2}}n_1)+x_1\frac{\partial}{\partial x_2}((n_1^2 + x_1^2n_2^2)^{\frac{p-2}{2}}x_1n_2) \\
		& = (p-2)|\nabla_{X}dist(x,\partial \Omega^+)|^{p-4}n_1n_2^2 x_1.
	\end{align*}
	Plugging the above expressions into inequality \eqref{Hardy_half-space} we arrive at 
		\begin{align*}
	\int_{\Omega^+} |\nabla_{X} f(x)|^p dx \geq& -(p-1)(|\gamma|^{\frac{p}{p-1}}+\gamma) \int_{\Omega^+} \frac{(n_1^2+x_1^2n_2^2)^{p/2}}{(x_1n_1+x_2n_2-d)^p}|f(x)|^p dx \\
	& +(p-2)\gamma\int_{\Omega^+} \frac{|\nabla_{X}dist(x,\partial \Omega^+)|^{p-4}n_1n_2^2 x_1}{(x_1n_1+x_2n_2-d)^{p-1}} |f(x)|^p dx, 
	\end{align*}
	 which proves inequality \eqref{Hardy_G_half-space}. 
	 If one of the cases $n(x)=(1,0)$ or $n(x)=(0,1)$ holds, then the last term of the above inequality vanishes, so that we get  
	 \begin{align}\label{1234}
	 \int_{\Omega^+} |\nabla_{X} f(x)|^p dx \geq -(p-1)(|\gamma|^{\frac{p}{p-1}}+\gamma) \int_{\Omega^+} \frac{(n_1^2+x_1^2n_2^2)^{p/2}}{(x_1n_1+x_2n_2-d)^p}|f(x)|^p dx. 
	 \end{align}
	 Then, we maximise above inequality by differentiating with respect to $\gamma$, so that  we have 
	 \begin{equation*}
	 	\frac{p}{p-1}|\gamma|^{\frac{1}{p-1}}+1=0,
	 \end{equation*}
	 which leads to
	 \begin{equation*}
	 	\gamma = - \left( \frac{p-1}{p}\right)^{p-1}.
	 \end{equation*}
	 By putting the value of $\gamma$ into inequality \eqref{1234}, we obtain inequality \eqref{Hardy_with_best_const}.  
\end{proof}
\begin{corollary}\label{cor_H1}
	Let $\Omega^+$ be a half-space on the Heisenberg group. Then for every function $f\in C_0^{\infty}(\Omega^+)$ and $p>1$, we have 
	\begin{equation}\label{H_half-space}
	\int_{\Omega^+}  |\nabla_H f(x)|^p dx \geq \left( \frac{p-1}{p}\right)^p \int_{\Omega^+} \frac{ |(n_1+2x_2n_3,n_2-2x_1n_3)|^p}{dist(x,\partial \Omega^+)^p}|f(x)|^pdx,
	\end{equation}
	where $dist(x,\partial \Omega^+)=\langle x , n(x)\rangle-d$ and $d\in \mathbb{R}$.
\end{corollary}
\begin{rem}
	Note that if we choose $n(x)=(0,0,1)$, $p=2$ and $d=0$ in inequality \eqref{H_half-space}, then we get
		\begin{equation}\label{LY}
	\int_{\Omega^+}  |\nabla_H f(x)|^2 dx \geq  \int_{\Omega^+} \frac{|x_1|^2+|x_2|^2}{|x_3|^2}|f(x)|^2dx.
	\end{equation}
	The Hardy inequality of the form \eqref{LY} in the half-space on the Heisenberg group was shown by
	Luan and Young \cite{LY}. 
\end{rem}
\begin{proof}[Proof of Corollary \ref{cor_H1}] 
		We begin the proof of Corollary \ref{cor_H1} by a simple computation such as
	\begin{align*}
	dist(x,\partial \Omega^+) &= x_1n_1 + x_2n_2+ x_3n_3-d,\\
	\nabla_{X} dist(x,\partial \Omega^+) &=(n_1+2x_2n_3,n_2-2x_1n_3), \\
	|\nabla_{X} dist(x,\partial \Omega^+)|^p &=((n_1+2x_2n_3)^2+(n_2-2x_1n_3)^2)^{p/2}.
	\end{align*} 
	Then we compute 
	\begin{align*}
	\mathcal{L}_pdist(x,\partial \Omega^+)  =&\nabla_{H} \cdot (|\nabla_{H} dist(x,\partial \Omega^+)|^{p-2}\nabla_{H}dist(x,\partial \Omega^+) ) \\
	=& X_1 ((n_1+2x_2n_3)^2+(n_2-2x_1n_3)^2)^{\frac{p-2}{2}}(n_1+2x_2n_3) \\
	+& X_2 ((n_1+2x_2n_3)^2+(n_2-2x_1n_3)^2)^{\frac{p-2}{2}}(n_2-2x_1n_3)\\
	=& -2(p-2)|\nabla_{H}dist(x,\partial \Omega^+)|^{p-4}(n_1+2x_2n_3)(n_2-2x_1n_3)n_3 \\
	& + 2(p-2)|\nabla_{H}dist(x,\partial \Omega^+)|^{p-4} (n_2-2x_1n_3)(n_1+2x_2n_3)n_3 \\
	=& 0.
	\end{align*}
	Plugging the above expressions into inequality \eqref{Hardy_half-space}, we arrive at 
		\begin{equation}\label{for_proof}
	\int_{\Omega^+}  |\nabla_H f(x)|^p dx \geq -(p-1)(|\gamma|^{\frac{p}{p-1}}+\gamma) \int_{\Omega^+} \frac{ |(n_1+2x_2n_3,n_2-2x_1n_3)|^p}{(x_1n_1+x_2n_2+x_3n_3-d)^p}|f(x)|^pdx,
	\end{equation}
	 which can be maximised by differentiating with respect to $\gamma$, then we have 
	  \begin{equation*}
	 \frac{p}{p-1}|\gamma|^{\frac{1}{p-1}}+1=0,
	 \end{equation*}
	 that leads to
	 \begin{equation*}
	 \gamma = - \left( \frac{p-1}{p}\right)^{p-1}.
	 \end{equation*}
	 By putting the value of $\gamma$ into inequality \eqref{for_proof}, we obtain inequality
	 	\begin{equation*}
	 \int_{\Omega^+}  |\nabla_H f(x)|^p dx \geq \left( \frac{p-1}{p}\right)^p \int_{\Omega^+} \frac{ |(n_1+2x_2n_3,n_2-2x_1n_3)|^p}{dist(x,\partial \Omega^+)^p}|f(x)|^pdx,
	 \end{equation*}
 which proves Corollary \ref{cor_H1}.  
\end{proof}
\begin{corollary}\label{cor_E1}
	Let $\Omega^+$ be a half-space on the Engel group $\mathbb{E}$. Then for every function $f\in C_0^{\infty}(\Omega^+)$, $\gamma \in \mathbb{R}$ and $p=2$, we have 
	\begin{align}\label{E_half-space}
	\int_{\Omega^+}  |\nabla_H f(x)|^2 dx \geq& - (|\gamma|^2+\gamma) \int_{\Omega^+} \frac{|\nabla_{H} dist(x,\partial \Omega^+)|^2}{dist(x,\partial \Omega^+)^2} |f(x)|^2 dx \\
	& + \frac{\gamma}{6} \int_{\Omega^+} \frac{x_2n_4}{dist(x,\partial \Omega^+)}|f(x)|^2 dx, \nonumber
	\end{align}
	where $dist(x,\partial \Omega^+)=\langle x , n(x)\rangle-d$ and $d\in \mathbb{R}$.
\end{corollary}
\begin{proof}[Proof of Corollary \ref{cor_E1}] 
	We begin the proof of Corollary \ref{cor_E1} by a simple computation such as
	\begin{align*}
		dist(x,\partial \Omega^+) &= x_1n_1+x_2n_2+ x_3n_3 +x_4n_4-d, \\
		\nabla_{H}	dist(x,\partial \Omega^+) &= \left( n_1 - \frac{x_2n_3}{2}- \frac{x_3n_4}{2} - \frac{x_1x_2n_4}{12}, n_2+\frac{x_1n_3}{2}+\frac{x_1^2n_4}{12} \right),\\
		|\nabla_{H}dist(x,\partial \Omega^+)|^2 &= \left( n_1 - \frac{x_2n_3}{2}- \frac{x_3n_4}{2} - \frac{x_1x_2n_4}{12}\right)^2+ \left(n_2+\frac{x_1n_3}{2}+\frac{x_1^2n_4}{12}\right)^2,
	\end{align*}
	and
	\begin{align*}
		\mathcal{L}	(dist(x,\partial \Omega^+) )=& \nabla_{H} \cdot \nabla_Hdist(x,\partial \Omega^+)\\
		=& \nabla_{H}\cdot \left( n_1 - \frac{x_2n_3}{2}- \frac{x_3n_4}{2} - \frac{x_1x_2n_4}{12}, n_2+\frac{x_1n_3}{2}+\frac{x_1^2n_4}{12} \right)\\
		=& X_1\left( n_1 - \frac{x_2n_3}{2}- \frac{x_3n_4}{2} - \frac{x_1x_2n_4}{12} \right)  + X_2\left( n_2+\frac{x_1n_3}{2}+\frac{x_1^2n_4}{12} \right) \\
		=&  \frac{x_2n_4}{6}.
	\end{align*}
	Plugging the above expressions into inequality \eqref{Hardy_half-space}, we get 
	\begin{align*}
	\int_{\Omega^+}  |\nabla_H f(x)|^2 dx \geq& - (|\gamma|^2+\gamma) \int_{\Omega^+} \frac{|\nabla_{H} dist(x,\partial \Omega^+)|^2}{dist(x,\partial \Omega^+)^2} |f(x)|^2 dx \\
	& + \frac{\gamma}{6} \int_{\Omega^+} \frac{x_2n_4}{dist(x,\partial \Omega^+)}|f(x)|^2 dx,
	\end{align*}
	which proves Corollary \ref{cor_E1}.
\end{proof}

\section{Proof of Main Results}\label{sec5}
The approach to prove the main results is based on the works  \cite{RSS_geom} and \cite{RSS_geo_H-S} (see, also \cite{RSS_Lp}-\cite{RS_local}).
For a vector field $g \in C^{\infty}(\Omega) $ we compute 
\begin{align*}
\int_{\Omega} { \rm div}_{X}g  |f(x)|^p dx &= -p \int_{\Omega}|f(x)|^{p-1}  \langle g, \nabla_{X} f(x) \rangle dx \\
& \leq p \left( \int_{\Omega} |\nabla_{X} f(x)|^p dx\right)^{\frac{1}{p}} \left( \int_{\Omega} |g|^{\frac{p}{p-1}}|f(x)|^p dx \right)^{\frac{p-1}{p}}\\
& \leq \int_{\Omega} |\nabla_{H} f(x)|^p dx + (p-1)\int_{\Omega} |g|^{\frac{p}{p-1}}|f(x)|^p dx.	\end{align*}
Here we have first used the divergence theorem, then we applied the H\"older inequality and the Young inequality.
By rearranging the above expression, we arrive at
\begin{equation}\label{1}
\int_{\Omega} |\nabla_{X} f(x)|^p dx \geq \int_{\Omega} ({\rm div}_{X}g-(p-1)|g|^{\frac{p}{p-1}} )|f(x)|^p dx.
\end{equation}
A suitable choice of the vector field $g$ in each special case is a key argument of our proofs.
\begin{proof}[Proof of Theorem \ref{G_starshaped}]
	Let us set  
	\begin{equation*}\label{23}
	g = \gamma \frac{|\nabla_{H}\langle Z(x),n(x)\rangle |^{p-2}}{|\langle Z(x),n(x)\rangle|^{p-1}}\nabla_{H} \langle Z(x),n(x)\rangle,
	\end{equation*}
	so that we have
	\begin{equation}\label{g_vecotr2}
	|g|^{\frac{p}{p-1}} =|\gamma|^{\frac{p}{p-1}} \frac{|\nabla_{H} \langle Z(x),n(x)\rangle|^p}{|\langle Z(x),n(x)\rangle|^p},
	\end{equation}
	and
	\begin{align}\label{div_g2}
	{ \rm div}_{H} g = \gamma \frac{\mathcal{L}_p (\langle Z(x),n(x)\rangle)}{|\langle Z(x),n(x)\rangle|^{p-1}} - \gamma(p-1)\frac{|\nabla_{H}\langle Z(x),n(x)\rangle|^p}{|\langle Z(x),n(x)\rangle|^p}.
	\end{align}
	Plugging the above expressions \eqref{g_vecotr2} and \eqref{div_g2} into inequality \eqref{1}, we get  
	\begin{align*}
	\int_{\Omega} |\nabla_H f(x)|^p dx \geq& -(p-1)(|\gamma|^{\frac{p}{p-1}}+\gamma) \int_{\Omega} \frac{|\nabla_H \langle Z(x),n(x)\rangle|^p}{|\langle Z(x),n(x)\rangle|^p}|f(x)|^p dx\\
	& + \gamma \int_{\Omega} \frac{\mathcal{L}_p(\langle Z(x),n(x)\rangle)}{|\langle Z(x),n(x)\rangle|^{p-1}} |f(x)|^p dx,
	\end{align*}
	which proves inequality \eqref{Hardy_stratified_starshaped}.
\end{proof}
\begin{proof}[Proof of Theorem \ref{M_half-space}]
	Let us take
	\begin{equation}\label{22}
	g = \gamma \frac{|\nabla_{X}dist(x,\partial \Omega^+) |^{p-2}}{dist(x,\partial \Omega^+)^{p-1}}\nabla_{X} dist(x,\partial \Omega^+),
	\end{equation}
	so that we have
	\begin{equation}\label{g_vecotr1}
	|g|^{\frac{p}{p-1}} =|\gamma|^{\frac{p}{p-1}} \frac{|\nabla_{X} dist(x,\partial \Omega^+)|^p}{dist(x,\partial \Omega^+)^p},
	\end{equation}
	and
	\begin{align}\label{div_g1}
	{ \rm div}_{X} g = \gamma \frac{\mathcal{L}_p dist(x,\partial \Omega^+)}{dist(x,\partial \Omega^+)^{p-1}} - \gamma(p-1)\frac{|\nabla_{X}dist(x,\partial \Omega^+)|^p}{dist(x,\partial \Omega^+)^p}.
	\end{align}
Combining expressions \eqref{g_vecotr1} and \eqref{div_g1} with inequality \eqref{1}, we obtain  
	\begin{align*}
	\int_{\Omega} |\nabla_X f(x)|^p dx \geq& -(p-1)(|\gamma|^{\frac{p}{p-1}}+\gamma) \int_{\Omega} \frac{|\nabla_X dist(x,\partial \Omega^+)|^p}{dist(x,\partial \Omega^+)^p}|f(x)|^p dx\\
	& + \gamma \int_{\Omega} \frac{\mathcal{L}_pdist(x,\partial \Omega^+)}{dist(x,\partial \Omega^+)^{p-1}} |f(x)|^p dx,
	\end{align*}
	which proves inequality \eqref{Hardy_half-space}.
\end{proof}


\begin{thebibliography}{NZW01}
\bibitem{BD3}
Baldi M., Dragoni F.:
\newblock Convexity and semiconvexity along vector fields. 
\newblock \textit{Calc. Var. PDE}, 42, no. 3-4, 405-427 (2011)
\bibitem{BD4}
Baldi M., Dragoni M.:
\newblock Subdifferential and properties of convex functions with respect to vector fields. 
\newblock \textit{J. Convex Anal.} 21, no. 3, 785-810 (2014)
	
		\bibitem{DGN}
	Danielli D. Garofalo N. and Nhieu D.M.:
	\newblock Notions of convexity in Carnot groups.
	\newblock \textit{Comm. Anal. Geom.} 11(2), 263-–341 (2003)
	
	\bibitem{DG98}
	Danielli D., Garofalo N.: 
Geometric properties of solutions to subelliptic equations in nilpotent Lie groups. Reaction diffusion systems (Trieste, 1995), 89–105, Lecture Notes in Pure and Appl. Math., 194, Dekker, New York, (1998)
	\bibitem{DG}
	Danielli D. and Garofalo N.:
	 Green functions in Carnot groups and the geometry of their level sets. Personal communication.
	\bibitem{DGS}
	Dragoni F., Garofalo N., and Salani P.: 
	Starshapedeness for fully-nonlinear equations in Carnot groups. {\it J. London Math. Soc.} (2) 00, 1--18 (2018)
	\bibitem{FR}
	Fischer V., Ruzhansky M.:
	\newblock \href{https://link.springer.com/book/10.1007%2F978-3-319-29558-9}{Quantization on nilpotent Lie groups.}
	\newblock \textit{Progress in Mathematics},  314, Birkh\"auser, (open access book) (2016)




	
	\bibitem{Garofalo}
	Garofalo N.:
	\newblock Geometric second derivative estimates in Carnot groups and convexity.
	\newblock \textit{Manuscripta Math.} 126, 353-–373 (2008)
	
	\bibitem{Hormander67}
	H\"ormander L.: 
	\newblock Hypoelliptic second order differential equations.
	\newblock \textit{Acta Math.} 119(1), 147-171 (1967)
	
	
	\bibitem{MR}
	Monti R., Rickly M.:
	\newblock Geodetically convex sets in the Heisenberg group. 
	\newblock \textit{J. Convex Anal.} 12, No. 1, 187-196 (2005) 
	
	\bibitem{RSS_geom}
	Ruzhansky M., Sabitbek B., Suragan D.:
	\newblock Subelliptic geometric Hardy type inequalities on half-spaces and convex domains.
	\newblock 	\href{https://arxiv.org/abs/1806.06226}{\textit{arXiv:1806.06226}} (2018)
	
	\bibitem{RSS_geo_H-S}
	Ruzhansky M., Sabitbek B., Suragan D.: 
	\newblock Geometric Hardy and Hardy-Sobolev inequalities on Heisenberg groups
	\newblock 	\href{https://arxiv.org/abs/1811.07181}{\textit{arXiv:1811.07181}} (2018)
	
	\bibitem{RSS_Lp}
	Ruzhansky M., Sabitbek B., Suragan D.:
	\newblock Weighted $L^p$-Hardy and $L^p$-Rellich inequalities with boundary terms on stratified Lie groups.
	\newblock  \textit{Rev. Mat. Complutense} 32, 19--35  (2019)
	
	

		\bibitem{RS_local}
	Ruzhansky M., Suragan D.:
	\newblock Local Hardy and Rellich inequalities for sums of squares. 
	\newblock  \textit{ Adv. Diff. Equations} 22, 505--540  (2017)
	
		\bibitem{RS_book}
	Ruzhansky M., Suragan D.:
	\newblock \href{https://www.springer.com/la/book/9783030028947}{Hardy inequalities on homogeneous groups.}
	\newblock \textit{Progress in Mathematics,} 327, Birkh\"auser, (open access book) (2019)
	
	\bibitem{LY}
	Luan, J., Yang, Q.: 
	\newblock A Hardy type inequality in the half-space on $\mathbb{R}^n$ and Heisenberg group. 
	\newblock \textit{J. Math. Anal. Appl.} 347(2), 645-651 (2008)
\end{thebibliography}
\end{document}